\documentclass[11pt,reqno]{amsart}

\usepackage[utf8]{inputenc}
\usepackage[T1]{fontenc}
\usepackage{lmodern}
\usepackage{amsmath,amssymb,amsthm,amscd,mathtools}
\usepackage[mathscr]{euscript}
\usepackage{mathrsfs}
\usepackage{enumitem}
\usepackage{xcolor}
\usepackage[colorlinks=true,linkcolor=blue,citecolor=black,urlcolor=black]{hyperref}
\usepackage{aliascnt}
\usepackage[nameinlink,noabbrev]{cleveref}

\theoremstyle{plain}
\newtheorem{theorem}{Theorem}[section]

\newaliascnt{lemma}{theorem}
\newtheorem{lemma}[lemma]{Lemma}
\aliascntresetthe{lemma}
\crefname{lemma}{lemma}{lemmas}
\Crefname{lemma}{Lemma}{Lemmas}

\newaliascnt{proposition}{theorem}
\newtheorem{proposition}[proposition]{Proposition}
\aliascntresetthe{proposition}
\crefname{proposition}{proposition}{propositions}
\Crefname{proposition}{Proposition}{Propositions}

\newaliascnt{corollary}{theorem}

\aliascntresetthe{corollary}
\crefname{corollary}{corollary}{corollaries}
\Crefname{corollary}{Corollary}{Corollaries}

\theoremstyle{definition}
\newaliascnt{definition}{theorem}
\newtheorem{definition}[definition]{Definition}
\aliascntresetthe{definition}
\crefname{definition}{definition}{definitions}
\Crefname{definition}{Definition}{Definitions}

\newaliascnt{remark}{theorem}
\newtheorem{remark}[remark]{Remark}
\aliascntresetthe{remark}
\crefname{remark}{remark}{remarks}
\Crefname{remark}{Remark}{Remarks}

\newcommand{\B}{\mathscr{B}}
\newcommand{\A}{\mathscr{A}}
\newcommand{\K}{\mathscr{K}}
\newcommand{\Q}{\mathscr{Q}}
\newcommand{\Jop}{\mathscr{J}}
\newcommand{\Ssep}{\mathscr{X}}
\newcommand{\F}{\mathbb F}
\newcommand{\C}{\mathbb C}
\newcommand{\N}{\mathbb N}

\newcommand{\norm}[1]{ \| #1 \| }

\DeclareMathOperator{\dens}{dens}
\DeclareMathOperator{\spanop}{span}



\begin{document}
\raggedbottom

\title[A Unital Banach Algebra Which Is Not a Calkin Algebra]{A Unital Banach Algebra Which Is Not a Calkin Algebra}

\author[A. Acuaviva]{Antonio Acuaviva}
\address{School of Mathematical Sciences,
Fylde College,
Lancaster University,
LA1 4YF,
United Kingdom} \email{ahacua@gmail.com}

\author[P. Acuaviva]{Pablo Acuaviva}
\address{Institute of Computer Science,
University of Bern,
Neubr\"uckstrasse 10,
3012 Bern,
Switzerland} \email{pablohacuaviva@gmail.com}

\date{\today}

\begin{abstract}
    Let $\A(X)$ and $\K(X)$ denote the ideals of approximable and compact operators on a Banach space $X$, respectively. We construct a unital Banach algebra $A$ of density character $\mathfrak c=2^{\aleph_0}$, with exactly one non-zero proper closed two-sided ideal, such that, for every Banach space $X$, the algebra $A$ is isomorphic to neither $\B(X)/\A(X)$ nor $\B(X)/\K(X)$. Thus $A$ is not a Calkin algebra under either of the two customary conventions.
\end{abstract}

\maketitle

\tableofcontents

\bigskip
\section{Introduction and organisation}\label{sec:introduction}

Given a Banach space $X$, let $\B(X)$ denote the Banach algebra of all bounded operators on $X$, let $\mathscr F(X)$ denote the finite-rank operators, and let
\begin{equation*}
    \A(X)=\overline{\mathscr F(X)}
\end{equation*}
be the ideal of approximable operators. The compact operators form a closed two-sided ideal $\K(X)$ of $\B(X)$, and
\begin{equation*}
    \A(X)\subseteq\K(X).
\end{equation*}
Equality holds whenever $X$ has the approximation property, but the inclusion may be strict in general. Accordingly, and following Horv\'ath and Kania \cite[Section~1]{HorvathKania2021}, we treat two customary definitions of the Calkin algebra simultaneously. We fix, once and for all, one of the assignments
\begin{equation}\label{eq:two-calkin-assignments}
    \Jop=\A
    \qquad\text{or}\qquad
    \Jop=\K,
\end{equation}
and put
\begin{equation*}
    \Q(X)=\B(X)/\Jop(X).
\end{equation*}
The choice in \eqref{eq:two-calkin-assignments} is used consistently for all Banach spaces. Every statement below involving $\Q$ is valid for either fixed choice, while our main theorem asserts that the same algebra fails to arise for both choices.

Two operators represent the same element of $\Q(X)$ precisely when their difference belongs to $\Jop(X)$, so the quotient captures the structure of the operators on $X$ up to the selected class of perturbations. Understanding this algebra can reveal structural features of the operators that a Banach space admits. Conversely, determining which Banach algebras can arise in this way helps to delineate the possible forms of operator algebras modulo small operators. For Hilbert spaces the two ideals coincide, and the compact-operator formulation goes back to Calkin's study of the ideals and congruences of $\B(\ell_2)$ \cite{Calkin1941}.

This leads naturally to the following realisation problem: given a unital Banach algebra $A$, does there exist a Banach space $X$ such that
\begin{equation*}
    A\cong\Q(X)
\end{equation*}
as Banach algebras? The problem was recorded in Tarbard's thesis and later discussed by Horv\'ath and Kania; see \cite[p.~134]{Tarbard2013} and \cite{HorvathKania2021}.

The scalar-plus-compact theorem of Argyros and Haydon provided the first striking example. They constructed an infinite-dimensional Banach space $X$ on which every operator is a scalar multiple of the identity plus a compact operator \cite{ArgyrosHaydon2011}, consequently obtaining the scalar field as a Calkin algebra. Their space has a Schauder basis, so the two quotient conventions agree. Later constructions revealed considerable flexibility. Motakis, Puglisi and Zisimopoulou realised $C(K)$ for every countable compact metric space $K$, and Motakis extended this to every compact metric space \cite{MotakisPuglisiZisimopoulou2016,Motakis2024}. Other examples include broad classes of diagonal scalar-plus-compact algebras \cite{MotakisPuglisiTolias2020}, infinite-dimensional reflexive Calkin algebras \cite{MotakisPelczarBarwacz2025}, and the unitisation of the noncommutative algebra $\K(c_0)$ \cite{MotakisPuglisi2025}. The realizing spaces in these results have the approximation property, so they also give examples under either convention.

Despite this range of positive results, it remained unknown whether every unital Banach algebra is isomorphic to the Calkin algebra of some Banach space. Horv\'ath and Kania obtained a partial negative result: their argument yields a simple unital AF $C^*$-algebra $D$ of density character $\mathfrak c$ which is not isomorphic to $\B(Y)/\K(Y)$ for any separable Banach space $Y$ \cite[Theorem~1.1(iii) and its proof]{HorvathKania2021}. Simplicity shows that this same $D$ is not isomorphic to $\B(Y)/\A(Y)$ for any separable $Y$ either; we record the argument in \Cref{thm:HK}. Their result did not rule out a representation over a nonseparable Banach space. The present paper removes this restriction, simultaneously for both quotient conventions.

\begin{theorem}\label{thm:main}
There is a unital Banach algebra $A$ of density character $\mathfrak c$, having exactly one non-zero proper closed two-sided ideal, such that, for every Banach space $X$,
\begin{equation*}
    A\not\cong\B(X)/\A(X)
    \qquad\text{and}\qquad
    A\not\cong\B(X)/\K(X)
\end{equation*}
as Banach algebras.
\end{theorem}

\subsection{Proof strategy and organisation}

Let $D$ be the simple unital AF $C^*$-algebra described above. On the Banach space
\begin{equation*}
    E_D=\ell_1(\N,D)
\end{equation*}
we consider the norm-closed algebra $J_D$ generated by the finite-support $D$-valued matrices, and set
\begin{equation*}
    A_D=\C I_{E_D}+J_D.
\end{equation*}
If $p$ is the first diagonal matrix unit, then
\begin{equation*}
    pA_Dp\cong D.
\end{equation*}
Moreover, the simplicity of $D$ implies that $J_D$ is topologically simple, while $J_D$ is nonunital. It follows that the closed two-sided ideals of $A_D$ are precisely
\begin{equation*}
    \{0\}\subsetneq J_D\subsetneq A_D.
\end{equation*}

Fix either choice of $\Jop$ in \eqref{eq:two-calkin-assignments}, and suppose that
\begin{equation*}
    \Phi\colon A_D\longrightarrow\Q(X)
\end{equation*}
is a Banach-algebra isomorphism. If $X$ is separable, the Calkin-corner lemma, \Cref{lem:calkin-corners}, identifies the corner determined by $\Phi(p)$ with $\Q(Y)$ for a separable complemented subspace $Y$ of $X$. Hence
\begin{equation*}
    D\cong pA_Dp\cong\Q(Y),
\end{equation*}
contrary to the choice of $D$.

Now suppose that $X$ is nonseparable. The operators with separable range determine a proper closed two-sided ideal
\begin{equation*}
    \mathcal R_X=\Ssep(X)/\Jop(X)
\end{equation*}
of $\Q(X)$. Since $A_D$ has only one non-zero proper closed two-sided ideal,
\begin{equation*}
    \Phi^{-1}(\mathcal R_X)=0
    \qquad\text{or}\qquad
    \Phi^{-1}(\mathcal R_X)=J_D.
\end{equation*}
In the second case, $\Phi(p)\in\mathcal R_X$, and the separable-range corner lemma, \Cref{lem:separable-range-corner}, gives a separable complemented subspace $Y$ such that
\begin{equation*}
    D\cong pA_Dp\cong\Q(Y),
\end{equation*}
which is impossible.

In the first case, the surjectivity of $\Phi$ gives $\mathcal R_X=0$, and therefore
\begin{equation*}
    \Ssep(X)=\Jop(X).
\end{equation*}
The remaining contradiction comes from the canonical matrix units in $A_D$. For each finite non-empty $F\subseteq\N$ and each scalar family $(a_j)_{j \in F}$, we have
\begin{equation*}
    \left\|\sum_{j \in F}a_j\varepsilon_{1j}\right\|
    \leq\sup_{j \in F}|a_j|,
    \qquad
    \sup_{i \in \N}\norm{\varepsilon_{i1}}\leq1.
\end{equation*}
Their images under $\Phi$ satisfy the same estimates up to a uniform constant. The matrix-unit obstruction proved in \Cref{sec:matrix-obstruction} produces a noncompact operator with separable range. Such an operator belongs to $\Ssep(X)$ but not to $\Jop(X)$, since $\Jop(X)\subseteq\K(X)$, contradicting the preceding equality.

In \Cref{sec:preliminaries} we fix notation, identify the coefficient algebra, and prove the lifting theorem for idempotents under both quotient conventions. In \Cref{sec:calkin-corners} we prove the corner lemmas and study the separable-range ideal. In \Cref{sec:matrix-obstruction} we establish the matrix-unit obstruction. Finally, in \Cref{sec:one-ideal-algebra} we construct $A_D$, determine its ideal lattice, and prove \Cref{thm:main}.

\begin{remark}[Real scalars]
We work throughout over the complex field, but \Cref{thm:main} has a real analogue. Let $A_{\mathbb R}$ denote the underlying real Banach algebra of the complex algebra $A$ constructed above. Every closed real two-sided ideal $I$ of $A_{\mathbb R}$ is automatically complex-linear: if $x\in I$, then
\begin{equation*}
    ix=(i1_A)x\in I.
\end{equation*}
Thus the closed real two-sided ideals of $A_{\mathbb R}$ are precisely the closed complex two-sided ideals of $A$. Moreover, $A_{\mathbb R}$ has the same underlying normed space as $A$, and hence the same density character.

Let $\Jop_{\mathbb R}$ denote either the real approximable-operator assignment or the real compact-operator assignment. Suppose, towards a contradiction, that there are a real Banach space $X$ and a real Banach-algebra isomorphism
\begin{equation*}
    A_{\mathbb R}\cong\Q_{\mathbb R}(X),
    \qquad
    \Q_{\mathbb R}(X)=\B_{\mathbb R}(X)/\Jop_{\mathbb R}(X).
\end{equation*}
Complexifying this isomorphism gives
\begin{equation*}
    (A_{\mathbb R})_{\mathbb C}\cong\bigl(\Q_{\mathbb R}(X)\bigr)_{\mathbb C}.
\end{equation*}
Standard complexification yields
\begin{equation*}
    (A_{\mathbb R})_{\mathbb C}\cong A\oplus\overline A,
\end{equation*}
where $\overline A$ denotes the conjugate complex Banach algebra. On the other hand, complexification commutes with quotients, and therefore
\begin{equation*}
    \bigl(\Q_{\mathbb R}(X)\bigr)_{\mathbb C}
    \cong
    \bigl(\B_{\mathbb R}(X)\bigr)_{\mathbb C}
    \big/
    \bigl(\Jop_{\mathbb R}(X)\bigr)_{\mathbb C} \cong
    \B(X_{\mathbb C})/\Jop(X_{\mathbb C})
    =\Q(X_{\mathbb C}),
\end{equation*}
where $\Jop$ is the corresponding complex operator ideal. Here
\begin{equation*}
    \bigl(\B_{\mathbb R}(X)\bigr)_{\mathbb C}\cong\B(X_{\mathbb C}),
\end{equation*}
and
\begin{equation*}
    \bigl(\Jop_{\mathbb R}(X)\bigr)_{\mathbb C}\cong\Jop(X_{\mathbb C}).
\end{equation*}
For $\Jop=\A$, the latter identification follows by complexifying finite-rank approximations, while for $\Jop=\K$ it is the standard compact-operator identity. Consequently, there is a complex Banach-algebra isomorphism
\begin{equation*}
    \Phi\colon A\oplus\overline A\longrightarrow\Q(X_{\mathbb C}).
\end{equation*}
Put
\begin{equation*}
    e=\Phi(1_A,0).
\end{equation*}
Then $e$ is an idempotent and
\begin{equation*}
    e\Q(X_{\mathbb C})e
    =\Phi\bigl((1_A,0)(A\oplus\overline A)(1_A,0)\bigr)
    =\Phi(A\oplus0)
    \cong A.
\end{equation*}
By \Cref{lem:calkin-corners}, there is a complemented closed complex subspace $Z\subseteq X_{\mathbb C}$ such that
\begin{equation*}
    e\Q(X_{\mathbb C})e\cong\Q(Z).
\end{equation*}
It follows that
\begin{equation*}
    A\cong\Q(Z),
\end{equation*}
contrary to \Cref{thm:main}. Since the choice of $\Jop_{\mathbb R}$ was arbitrary, $A_{\mathbb R}$ is not isomorphic to either real Calkin quotient of any real Banach space.
\end{remark}
\bigskip
\section{Notation and preliminary results}\label{sec:preliminaries}

We use standard notation and conventions, unless explicitly stated other\-wise. We write $\N=\{1,2,\ldots\}$, and denote the closed unit ball of a Banach space $X$ by $B_X$. By an \emph{operator}, we mean a bounded linear map. For Banach spaces $X$ and $Y$, write $\B(X,Y)$ for the space of operators from $X$ to $Y$, and write $\B(X)=\B(X,X)$. The identity operator on $X$ is denoted by $I_X$, or simply by $I$ when the underlying space is clear. If $A$ is a unital Banach algebra, its identity is denoted by $1_A$.

We write $\mathscr F(X,Y)$ for the finite-rank operators from $X$ to $Y$ and
\begin{equation*}
    \A(X,Y)=\overline{\mathscr F(X,Y)}
\end{equation*}
for the approximable operators. We also write $\K(X,Y)$ for the compact operators from $X$ to $Y$. As usual, the second variable is omitted when $X=Y$. Throughout the paper, $\Jop$ denotes either $\A$ or $\K$, with one choice fixed consistently for all pairs of Banach spaces, and
\begin{equation*}
    \Q(X)=\B(X)/\Jop(X).
\end{equation*}
The corresponding quotient map will usually be denoted by
\begin{equation*}
    q\colon\B(X)\longrightarrow\Q(X).
\end{equation*}
In either case,
\begin{equation}\label{eq:J-contained-in-K}
    \Jop(X,Y)\subseteq\K(X,Y).
\end{equation}
An isomorphism of Banach algebras means a bounded complex-linear algebra isomorphism with bounded inverse.

We shall also use the following terminology.

\begin{definition}\label{def:density-simple}
Let $M$ be a topological space and let $A$ be a Banach algebra.
\begin{enumerate}[label=(\alph*)]
    \item The \emph{density character} of $M$, denoted by $\dens(M)$, is the least cardinality of a dense subset of $M$.
    \item The algebra $A$ is \emph{topologically simple} if its only closed two-sided ideals are $0$ and $A$.
\end{enumerate}
\end{definition}

The coefficient algebra in our construction must be topologically simple and must not be a Calkin algebra of a separable Banach space under either convention. The following consequence of Horv\'ath and Kania gives one algebra that works simultaneously for both choices.

\begin{theorem}[Horv\'ath--Kania]\label{thm:HK}
There is a simple unital AF $C^*$-algebra $D$ of density character $\mathfrak c$ such that, for every separable Banach space $Y$,
\begin{equation*}
    D\not\cong\B(Y)/\A(Y)
    \qquad\text{and}\qquad
    D\not\cong\B(Y)/\K(Y)
\end{equation*}
as Banach algebras.
\end{theorem}

\begin{proof}
Apply \cite[Theorem~1.1(iii)]{HorvathKania2021} to the compact-operator assignment with $\lambda=\aleph_0$. In the notation used in the proof of that theorem, $\kappa=2^\lambda$, and the AF $C^*$-algebra in part~\textup{(iii)} is chosen to have density character $\kappa$. Hence its density character is
\begin{equation*}
    2^{\aleph_0}=\mathfrak c,
\end{equation*}
using the family of pairwise non-isomorphic simple AF algebras constructed by Farah and Katsura \cite{FarahKatsura2015}. The theorem ensures that this algebra $D$ is not isomorphic to $\B(Y)/\K(Y)$ for any non-zero separable Banach space $Y$. The conclusion is also clear for the zero space.

It remains to check that the same $D$ works for approximable operators. Suppose that
\begin{equation*}
    D\cong\B(Y)/\A(Y)
\end{equation*}
for some separable Banach space $Y$. The quotient is zero when $Y$ is finite-dimensional, so $Y$ must be infinite-dimensional. The canonical surjection
\begin{equation*}
    \B(Y)/\A(Y)\longrightarrow\B(Y)/\K(Y)
\end{equation*}
has kernel $\K(Y)/\A(Y)$. This kernel is a proper closed two-sided ideal because the codomain is non-zero. Its inverse image in $D$ under the supposed isomorphism must therefore be zero, since every closed two-sided ideal in the simple $C^*$-algebra $D$ is either zero or all of $D$; see, for example, \cite[Theorem~3.1.3]{Murphy1990}. Hence
\begin{equation*}
    \A(Y)=\K(Y),
\end{equation*}
which would give $D\cong\B(Y)/\K(Y)$, a contradiction.
\end{proof}

\begin{remark}\label{rem:Cstar-simple}
The algebra $D$ in \Cref{thm:HK} is topologically simple when regarded merely as a Banach algebra. Indeed, every closed two-sided ideal in a $C^*$-algebra is self-adjoint by the result cited in the preceding proof. Thus simplicity as a $C^*$-algebra rules out non-zero proper closed two-sided ideals even after the involution is forgotten.
\end{remark}

To identify a corner of a Calkin algebra, we need to lift its defining idempotent to an idempotent operator without leaving the selected operator ideal. The following elementary spectral argument gives the required lifting result under both quotient conventions.

\begin{lemma}[Lifting idempotents from a Calkin algebra]\label{lem:essential-idempotent-lift}
Let $X$ be a Banach space and let $T\in\B(X)$. If
\begin{equation*}
    T^2-T\in\Jop(X),
\end{equation*}
then there is an idempotent $P\in\B(X)$ such that
\begin{equation*}
    P-T\in\Jop(X).
\end{equation*}
\end{lemma}

\begin{proof}
    If $T\in\Jop(X)$, take $P=0$. If $I_X-T\in\Jop(X)$, take $P=I_X$. We may therefore suppose that
    \begin{equation*}
        e=q(T)
    \end{equation*}
    is a non-zero idempotent distinct from the identity of $\Q(X)$.
    
    Put
    \begin{equation*}
        K=T^2-T\in\Jop(X).
    \end{equation*}
    Since $\Jop(X)\subseteq\K(X)$, the operator $K$ is compact. Hence $\sigma_{\B(X)}(K)$ is countable, with zero as its only possible accumulation point. By the polynomial spectral mapping theorem,
    \begin{equation*}
        \sigma_{\B(X)}(K)=\{\lambda(\lambda-1):\lambda\in\sigma_{\B(X)}(T)\}.
    \end{equation*}
    The polynomial $z\mapsto z(z-1)$ has finite fibres, so it follows that $\sigma_{\B(X)}(T)$ is countable.
    
    Since $e$ is a non-zero idempotent distinct from the identity,
    \begin{equation*}
        \sigma_{\Q(X)}(e)=\{0,1\}.
    \end{equation*}
    Spectral inclusion under the quotient map therefore gives
    \begin{equation*}
        \{0,1\}\subseteq\sigma_{\B(X)}(T).
    \end{equation*}
    
    Since $\sigma_{\B(X)}(T)$ is countable, we may choose $r\in(0,1)$ such that
    \begin{equation*}
        \sigma_{\B(X)}(T)\cap\{z\in\C:|z|=r\}=\varnothing.
    \end{equation*}
    Let $\gamma_r$ denote the positively oriented circle $\{z\in\C:|z|=r\}$, and define
    \begin{equation*}
        E=\frac{1}{2\pi i}\int_{\gamma_r}(zI_X-T)^{-1}\,dz.
    \end{equation*}
    By the Riesz functional calculus, $E$ is an idempotent. Consequently,
    \begin{equation*}
        P=I_X-E
    \end{equation*}
    is also an idempotent.
    
    It remains to show that $P-T\in\Jop(X)$. For every $z\in\gamma_r$,
    \begin{equation*}
        q\bigl((zI_X-T)^{-1}\bigr)=(z1_{\Q(X)}-e)^{-1}.
    \end{equation*}
    Since $q$ is continuous and linear, it may be applied under the contour integral, giving
    \begin{equation*}
        q(E)=\frac{1}{2\pi i}\int_{\gamma_r}(z1_{\Q(X)}-e)^{-1}\,dz.
    \end{equation*}
    For $z\notin\{0,1\}$, the identity $e^2=e$ gives
    \begin{equation*}
        (z1_{\Q(X)}-e)^{-1}=\frac{1}{z}(1_{\Q(X)}-e)+\frac{1}{z-1}e.
    \end{equation*}
    The circle $\gamma_r$ surrounds zero but not one. Hence
    \begin{equation*}
        q(E)=1_{\Q(X)}-e.
    \end{equation*}
    It follows that
    \begin{equation*}
        q(P)=1_{\Q(X)}-q(E)=e=q(T).
    \end{equation*}
    Therefore
    \begin{equation*}
        P-T\in\ker q=\Jop(X),
    \end{equation*}
    as required.
\end{proof}
\bigskip
\section{Calkin corners and separable-range ideals}\label{sec:calkin-corners}

We begin with the corner lemma used in the proof of \Cref{thm:main}. If $P\in\B(X)$ is an idempotent, then $PX$ is a complemented closed subspace of $X$, and the operator corner $P\B(X)P$ is naturally isomorphic to $\B(PX)$. The same identification persists after passage to either Calkin quotient.

\begin{lemma}[Corners of Calkin algebras]\label{lem:calkin-corners}
Let $X$ be a Banach space and let $e\in\Q(X)$ be an idempotent. Then there is a complemented closed subspace $Y\subseteq X$ such that
\begin{equation*}
    e\Q(X)e\cong\Q(Y).
\end{equation*}
If $X$ is separable, then $Y$ is separable.
\end{lemma}

\begin{proof}
Choose $T\in\B(X)$ with $q(T)=e$. Since $e$ is an idempotent,
\begin{equation*}
    T^2-T\in\Jop(X).
\end{equation*}
By \Cref{lem:essential-idempotent-lift}, there is an idempotent $P\in\B(X)$ such that $q(P)=e$. Put
\begin{equation*}
    Y=PX.
\end{equation*}
Then $Y$ is a complemented closed subspace of $X$. The corner $P\B(X)P$ is closed, since it is the range of the bounded projection
\begin{equation*}
    S\longmapsto PSP
\end{equation*}
on $\B(X)$.

Let $\iota\colon Y\longrightarrow X$ denote the inclusion map. The restriction map
\begin{equation*}
    \rho\colon P\B(X)P\longrightarrow\B(Y),
    \qquad
    \rho(S)=S|_Y,
\end{equation*}
is a Banach-algebra isomorphism. Its inverse sends $R\in\B(Y)$ to $\iota RP\in P\B(X)P$. We claim that
\begin{equation}\label{eq:corner-selected-ideal}
    \rho\bigl(P\B(X)P\cap\Jop(X)\bigr)=\Jop(Y).
\end{equation}

For $\Jop=\K$, this follows because compactness is preserved under restriction and composition with bounded operators. Explicitly, if $S\in P\B(X)P$ and $S|_Y$ is compact, then
\begin{equation*}
    S=\iota(S|_Y)P
\end{equation*}
is compact on $X$, and the converse follows by restriction.

For $\Jop=\A$, let $S\in P\B(X)P\cap\A(X)$ and choose finite-rank operators $F_n\in\mathscr F(X)$ such that $F_n\to S$. Then
\begin{equation*}
    PF_nP\longrightarrow PSP=S,
\end{equation*}
and each $(PF_nP)|_Y$ is a finite-rank operator on $Y$. Hence $S|_Y\in\A(Y)$. Conversely, if $R\in\A(Y)$ and $R_n\in\mathscr F(Y)$ converges to $R$, then
\begin{equation*}
    \iota R_nP\longrightarrow\iota RP
\end{equation*}
in $\B(X)$, and every $\iota R_nP$ has finite rank. Thus $\iota RP\in\A(X)$, proving \eqref{eq:corner-selected-ideal} in the second case as well.

The restriction of $q$ to $P\B(X)P$ maps onto $e\Q(X)e$. To see this, let $a\in e\Q(X)e$ and choose $S\in\B(X)$ with $q(S)=a$. Then
\begin{equation*}
    a=eq(S)e=q(PSP).
\end{equation*}
The kernel of the restricted map is $P\B(X)P\cap\Jop(X)$. Hence, by \eqref{eq:corner-selected-ideal},
\begin{equation*}
\begin{aligned}
    e\Q(X)e
    &\cong P\B(X)P/(P\B(X)P\cap\Jop(X))\\
    &\cong \B(Y)/\Jop(Y)
    =\Q(Y).
\end{aligned}
\end{equation*}
Here $e\Q(X)e$ is closed, since it is the range of the bounded projection $a\mapsto eae$ on $\Q(X)$. Thus the displayed identifications are Banach-algebra isomorphisms. If $X$ is separable, then its closed subspace $Y$ is separable.
\end{proof}

For the nonseparable part of the argument, we use the ideal of operators with separable range. For a Banach space $X$, set
\begin{equation*}
    \Ssep(X)=\{T\in\B(X):\overline{T(X)}\text{ is separable}\}.
\end{equation*}
We record the properties that will be needed later.

\begin{lemma}[The separable-range ideal]\label{lem:separable-range-ideal}
For every Banach space $X$, the set $\Ssep(X)$ is a closed two-sided ideal of $\B(X)$, and
\begin{equation*}
    \Jop(X)\subseteq\K(X)\subseteq\Ssep(X).
\end{equation*}
If $X$ is nonseparable, then $\Ssep(X)$ is proper and, consequently,
\begin{equation*}
    \mathcal R_X=\Ssep(X)/\Jop(X)
\end{equation*}
is a proper closed two-sided ideal of $\Q(X)$.
\end{lemma}

\begin{proof}
The set $\Ssep(X)$ is a linear subspace of $\B(X)$. Indeed, if $S,T\in\Ssep(X)$, then the closed linear span of $S(X)\cup T(X)$ is separable and contains the range of every linear combination of $S$ and $T$.

Let $T\in\Ssep(X)$ and let $A,B\in\B(X)$. The range of $AT$ is contained in the image under $A$ of the separable space $\overline{T(X)}$, while
\begin{equation*}
    TB(X)\subseteq T(X).
\end{equation*}
Thus both $AT$ and $TB$ have separable range, so $\Ssep(X)$ is a two-sided ideal.

Every compact operator belongs to $\Ssep(X)$. Indeed, if $K\in\K(X)$, then $\overline{K(B_X)}$ is compact and therefore separable, while
\begin{equation*}
    K(X)=\bigcup_{n\in\N}nK(B_X).
\end{equation*}
It follows that $\overline{K(X)}$ is separable. Together with \eqref{eq:J-contained-in-K}, this proves the displayed inclusions.

To see that the ideal is closed, suppose that $(T_n)_{n\in\N}\subseteq\Ssep(X)$ and $T_n\to T$ in norm. The closed linear span
\begin{equation*}
    M=\overline{\spanop\!\left(\bigcup_{n\in\N}T_n(X)\right)}
\end{equation*}
is separable. For every $x\in X$, the sequence $(T_nx)_{n\in\N}$ converges to $Tx$, so $Tx\in M$. Hence $T\in\Ssep(X)$.

Suppose now that $X$ is nonseparable. Then $I_X\notin\Ssep(X)$, and therefore $\Ssep(X)$ is proper. Since
\begin{equation*}
    q^{-1}(\mathcal R_X)=\Ssep(X),
\end{equation*}
its image $\mathcal R_X$ in $\Q(X)$ is a closed two-sided ideal. It is proper as well. Indeed, if $q(I_X)\in\mathcal R_X$, then there is $S\in\Ssep(X)$ such that
\begin{equation*}
    I_X-S\in\Jop(X)\subseteq\Ssep(X).
\end{equation*}
This would imply $I_X\in\Ssep(X)$, a contradiction.
\end{proof}

When an idempotent belongs to $\mathcal R_X$, the complemented subspace given by the corner lemma is separable even if the ambient space is not.

\begin{lemma}[Corners inside the separable-range ideal]\label{lem:separable-range-corner}
Let $X$ be a Banach space and let
\begin{equation*}
    e\in\Ssep(X)/\Jop(X)\subseteq\Q(X)
\end{equation*}
be an idempotent. Then there is a separable complemented closed subspace $Y\subseteq X$ such that
\begin{equation*}
    e\Q(X)e\cong\Q(Y).
\end{equation*}
\end{lemma}

\begin{proof}
Choose $T\in\Ssep(X)$ with $q(T)=e$. By \Cref{lem:essential-idempotent-lift}, there is an idempotent $P\in\B(X)$ such that
\begin{equation*}
    P-T\in\Jop(X).
\end{equation*}
Since $\Jop(X)\subseteq\Ssep(X)$ and $\Ssep(X)$ is linear, it follows that
\begin{equation*}
    P=T+(P-T)\in\Ssep(X).
\end{equation*}
The range
\begin{equation*}
    Y=PX
\end{equation*}
is closed because $P$ is an idempotent, and it is separable because $P$ has separable range. It is complemented by $P$. The computation in \Cref{lem:calkin-corners} now gives
\begin{equation*}
    e\Q(X)e\cong\Q(Y).
\end{equation*}
\end{proof}

\bigskip
\section{The matrix-unit obstruction}\label{sec:matrix-obstruction}

We prove the operator-theoretic obstruction used in the nonseparable case. If
\begin{equation*}
    \Ssep(X)=\Jop(X),
\end{equation*}
then a countable system of matrix units with a $c_0$-type first-row estimate and a bounded first column cannot lie in $\Q(X)$, under either quotient convention. The argument follows the matrix $\ell_1$ part of \cite[Lemmas~6.2, 6.3 and 6.9]{AcuavivaAcuaviva2026Wild}. For completeness, we reproduce the full proof here.

We lift the matrix units to $\B(X)$ and pass to a quotient on which all matrix-unit relations become exact. Two elementary lemmas are needed. The first concerns operators whose difference has range in the subspace being quotiented out.

\begin{lemma}[Quotient correction]\label{lem:quotient-correction}
Let $X$ be a Banach space, let $N\subseteq X$ be a closed subspace, and let
\begin{equation*}
    \pi\colon X\longrightarrow X/N
\end{equation*}
be the quotient map. Suppose that $U,K \in \B(X)$, $U(N)\subseteq N$, and $K(X)\subseteq N$. Then $U+K$ leaves $N$ invariant, $U$ and $U+K$ induce the same operator on $X/N$, and the induced operator $\overline U$ satisfies
\begin{equation*}
    \norm{\overline U}\leq\norm{U+K}.
\end{equation*}
\end{lemma}

\begin{proof}
Since $K(X)\subseteq N$, we also have $K(N)\subseteq N$, and therefore $U+K$ leaves $N$ invariant. For every $x \in X$,
\begin{equation*}
    \pi((U+K)x)=\pi(Ux)+\pi(Kx)=\pi(Ux).
\end{equation*}
Thus $U$ and $U+K$ induce the same operator on $X/N$. The norm of an induced operator is bounded by the norm of every operator inducing it, which gives the stated estimate.
\end{proof}

We shall place the ranges of the compact operators in one separable subspace which is invariant under the chosen lifts. The following lemma constructs such a subspace for any countable families of bounded and compact operators.

\begin{lemma}[A separable invariant subspace]\label{lem:absorb}
Let $X$ be a Banach space, let $(R_n)_{n \in \N}\subseteq\B(X)$, and let $(K_m)_{m \in \N}\subseteq\K(X)$. Then there is a separable closed subspace $N\subseteq X$ such that
\begin{equation*}
    K_m(X)\subseteq N\qquad(m \in \N)
\end{equation*}
and
\begin{equation*}
    R_n(N)\subseteq N\qquad(n \in \N).
\end{equation*}
\end{lemma}

\begin{proof}
For each $m \in \N$, the closed linear span of $K_m(X)$ is separable. Indeed, $\overline{K_m(B_X)}$ is compact and hence separable, and
\begin{equation*}
    K_m(X)=\bigcup_{r=1}^{\infty}rK_m(B_X).
\end{equation*}
Let $N_0$ be the closed linear span of
\begin{equation*}
    \bigcup_{m \in \N}K_m(X).
\end{equation*}
Then $N_0$ is separable.

Let $N$ be the closed linear span of all vectors of the form
\begin{equation*}
    R_{i_1}R_{i_2}\cdots R_{i_k}x,
\end{equation*}
where $k\geq 0$, $(i_r)_{r=1}^{k}$ is a finite sequence in $\N$, and $x \in N_0$. There are only countably many finite words in the family $(R_n)_{n \in \N}$, and the image of a separable space under an operator is separable. Hence $N$ is separable. By construction, it contains every $K_m(X)$ and is invariant under every $R_n$.
\end{proof}

We now state the obstruction. The first-row estimate is a $c_0$-type estimate on the scalar coefficients, while the first-column estimate keeps the vectors generated from a fixed non-zero corner uniformly bounded.

\begin{proposition}[$\ell_1$-row obstruction]\label{prop:matrix-obstruction}
Let $X$ be nonseparable and suppose that $\Ssep(X)=\Jop(X)$. Then $\Q(X)$ cannot contain a family $(e_{ij})_{i,j \in \N}$ satisfying the following conditions.
\begin{enumerate}[label=\textup{(\roman*)}]
    \item\label{item:matrix-obstruction-relations} $e_{ij}e_{kl}=\delta_{jk}e_{il}$ for all $i,j,k,l \in \N$;
    \item\label{item:matrix-obstruction-nonzero} $e_{11}\ne 0$;
    \item\label{item:matrix-obstruction-row} there is a constant $C_r<\infty$ such that
    \begin{equation*}
        \left\|\sum_{j \in F}a_je_{1j}\right\|\leq C_r\sup_{j \in F}|a_j|
    \end{equation*}
    for every finite non-empty $F\subseteq \N$ and every scalar family $(a_j)_{j \in F}$;
    \item\label{item:matrix-obstruction-column} there is a constant $C_c<\infty$ such that
    \begin{equation*}
        \sup_{i \in \N}\norm{e_{i1}}\leq C_c.
    \end{equation*}
\end{enumerate}
\end{proposition}

\begin{proof}
Suppose, towards a contradiction, that such a family $(e_{ij})_{i,j \in \N}$ exists. Let
\begin{equation*}
    q\colon\B(X)\longrightarrow\Q(X)
\end{equation*}
be the quotient map, and choose a family of lifts $(E_{ij})_{i,j \in \N}\subseteq\B(X)$ satisfying
\begin{equation*}
    q(E_{ij})=e_{ij}\qquad(i,j \in \N).
\end{equation*}
The matrix-unit relations hold modulo $\Jop(X)$. Thus, for all $i,j,k,l \in \N$,
\begin{equation*}
    D_{ij,kl}=E_{ij}E_{kl}-\delta_{jk}E_{il} \in \Jop(X).
\end{equation*}

We choose corrections from the selected quotient ideal so that the row and column estimates hold for suitable representatives in $\B(X)$. Let $\F_0=\mathbb Q+i\mathbb Q$, a countable dense subfield of $\C$. For every finite non-empty set $F\subseteq \N$ and every tuple $a=(a_j)_{j \in F} \in \F_0^F$, put
\begin{equation*}
    U_{F,a}=\sum_{j \in F}a_jE_{1j}.
\end{equation*}
Put
\begin{equation*}
    \alpha_{F,a}=\sup_{j \in F}|a_j|.
\end{equation*}
If $\alpha_{F,a}=0$, set $L_{F,a}=0$. Otherwise,
\begin{equation*}
    \norm{q(U_{F,a})}\leq C_r\alpha_{F,a},
\end{equation*}
by \ref{item:matrix-obstruction-row}. Since the quotient norm is the infimum of $\norm{U_{F,a}+L}$ over $L\in\Jop(X)$, we may choose $L_{F,a}\in\Jop(X)$ such that
\begin{equation}
\label{eq:row-correction}
    \norm{U_{F,a}+L_{F,a}}\leq(C_r+1)\alpha_{F,a}.
\end{equation}
Similarly, $\norm{q(E_{i1})}=\norm{e_{i1}}\leq C_c$ by \ref{item:matrix-obstruction-column}. Hence, for every $i \in \N$, there is $H_i\in\Jop(X)$ such that
\begin{equation}
\label{eq:column-correction}
    \norm{E_{i1}+H_i}\leq C_c+1.
\end{equation}

The defect and correction operators
\begin{equation*}
\begin{gathered}
    D_{ij,kl}\quad(i,j,k,l \in \N),\\
    L_{F,a}\quad(F\subseteq \N\text{ finite and non-empty},\ a \in \F_0^F),\\
    H_i\quad(i \in \N)
\end{gathered}
\end{equation*}
belong to $\Jop(X)\subseteq\K(X)$ and form a countable family. Apply \Cref{lem:absorb} to this family and to the countable family $(E_{ij})_{i,j \in \N}$. We obtain a separable closed subspace $N\subseteq X$ such that every $E_{ij}$ leaves $N$ invariant and the range of every defect and correction operator listed above is contained in $N$.

Put
\begin{equation*}
    Z=X/N,
\end{equation*}
and let $\pi\colon X\longrightarrow Z$ be the quotient map. Since each $E_{ij}$ leaves $N$ invariant, it induces a well-defined bounded operator $\overline E_{ij}\in\B(Z)$ as follows: given $z\in Z$, choose $x\in X$ such that $z=\pi x$ and set
\begin{equation*}
    \overline E_{ij}z=\pi(E_{ij}x).
\end{equation*}
Indeed, if $\pi x=\pi y$, then $x-y\in N$, so $E_{ij}(x-y)\in N$ and therefore $\pi(E_{ij}x)=\pi(E_{ij}y)$. Moreover, $\|\overline E_{ij}\|\leq\|E_{ij}\|$. Since $D_{ij,kl}(X)\subseteq N$, for every $x\in X$ we have
\begin{equation*}
    \bigl(\overline E_{ij}\overline E_{kl}-\delta_{jk}\overline E_{il}\bigr)(\pi x)=\pi\bigl((E_{ij}E_{kl}-\delta_{jk}E_{il})x\bigr)=\pi(D_{ij,kl}x)=0.
\end{equation*}
Since $\pi$ is surjective, the induced operators satisfy the exact matrix-unit relations
\begin{equation}
\label{eq:exact-matrix-units}
    \overline E_{ij}\overline E_{kl}=\delta_{jk}\overline E_{il}.
\end{equation}

Each correction operator has range in $N$ and therefore induces the zero operator on $Z$. By \Cref{lem:quotient-correction}, \eqref{eq:row-correction} gives
\begin{equation}
\label{eq:induced-row-bound-rational}
    \left\|\sum_{j \in F}a_j\overline E_{1j}\right\|\leq(C_r+1)\sup_{j \in F}|a_j|
\end{equation}
whenever $F\subseteq \N$ is finite and non-empty and $(a_j)_{j \in F} \in \F_0^F$. For a fixed finite non-empty set $F$, the map
\begin{equation*}
    (a_j)_{j \in F}\longmapsto\sum_{j \in F}a_j\overline E_{1j}
\end{equation*}
is continuous. Since $\F_0^F$ is dense in $\C^F$, \eqref{eq:induced-row-bound-rational} extends to
\begin{equation}
\label{eq:induced-row-bound}
    \left\|\sum_{j \in F}a_j\overline E_{1j}\right\|\leq(C_r+1)\sup_{j \in F}|a_j|
\end{equation}
for every finite non-empty $F\subseteq \N$ and every scalar family $(a_j)_{j \in F}$. Similarly, \eqref{eq:column-correction} gives
\begin{equation}
\label{eq:induced-column-bound}
    \sup_{i \in \N}\norm{\overline E_{i1}}\leq C_c+1.
\end{equation}

We claim that
\begin{equation*}
    \overline E_{11}\ne 0.
\end{equation*}
Indeed, if $\overline E_{11}=0$, then $E_{11}(X)\subseteq N$, so $E_{11}$ has separable range. The hypothesis $\Ssep(X)=\Jop(X)$ would then imply that $E_{11}\in\Jop(X)$, and hence
\begin{equation*}
    e_{11}=q(E_{11})=0,
\end{equation*}
contrary to \ref{item:matrix-obstruction-nonzero}.

Choose $y \in Z$ such that $\overline E_{11}y\ne 0$, and put
\begin{equation*}
    z_0=\overline E_{11}y.
\end{equation*}
Then $z_0\ne 0$ and, by \eqref{eq:exact-matrix-units},
\begin{equation*}
    \overline E_{11}z_0=z_0.
\end{equation*}
By the Hahn--Banach theorem, choose $\varphi \in Z^*$ with $\varphi(z_0)=1$.

The first-row estimate gives an $\ell_1$-valued operator. Indeed, define
\begin{equation*}
    \Theta x=\bigl(\varphi(\overline E_{1j}\pi x)\bigr)_{j \in \N}\qquad(x \in X).
\end{equation*}
For $x \in X$ and $n \in \N$, \eqref{eq:induced-row-bound} gives
\begin{equation*}
\begin{aligned}
    \sum_{j=1}^{n}|\varphi(\overline E_{1j}\pi x)|
    &=\sup_{\substack{(a_j)_{j=1}^{n} \in \C^n\\ |a_j|\leq1\ (1\leq j\leq n)}}
      \left|\varphi\left(\sum_{j=1}^{n}a_j\overline E_{1j}\pi x\right)\right|\\
    &\leq(C_r+1)\norm{\varphi}\,\norm{\pi x}\leq(C_r+1)\norm{\varphi}\,\norm{x}.
\end{aligned}
\end{equation*}
Taking the supremum over $n$ shows that $\Theta x \in \ell_1$ and that
\begin{equation*}
    \norm{\Theta x}_{\ell_1}\leq(C_r+1)\norm{\varphi}\,\norm{x}.
\end{equation*}
Thus $\Theta\colon X\to\ell_1$ is bounded.

The first-column estimate gives a bounded sequence which $\Theta$ maps onto the unit vector basis. For $i \in \N$, put
\begin{equation*}
    z_i=\overline E_{i1}z_0.
\end{equation*}
By \eqref{eq:induced-column-bound},
\begin{equation*}
    \norm{z_i}\leq(C_c+1)\norm{z_0}
    \qquad(i \in \N),
\end{equation*}
so the sequence $(z_i)_{i \in \N}$ is bounded. Choose a family of lifts $(x_i)_{i \in \N}\subseteq X$ such that
\begin{equation*}
    \pi x_i=z_i
    \qquad\text{and}\qquad
    \norm{x_i}\leq\norm{z_i}+1.
\end{equation*}
Then $(x_i)_{i \in \N}$ is bounded. For $i,k \in \N$, \eqref{eq:exact-matrix-units} gives
\begin{equation*}
\begin{aligned}
    \varphi(\overline E_{1k}z_i)
    &=\varphi(\overline E_{1k}\overline E_{i1}z_0)\\
    &=\delta_{ki}\varphi(\overline E_{11}z_0)
    =\delta_{ki}.
\end{aligned}
\end{equation*}
Consequently, $\Theta x_i=e_i$ for every $i\in\N$, where $(e_i)_{i\in\N}$ denotes the unit vector basis of $\ell_1$. Since $(x_i)_{i\in\N}$ is bounded, for every $(a_i)_{i\in\N}\in\ell_1$ the series $\sum_{i=1}^{\infty}a_ix_i$ converges absolutely, and the formula
\begin{equation*}
    T((a_i)_{i\in\N})=\sum_{i=1}^{\infty}a_ix_i
\end{equation*}
defines a bounded operator $T\colon\ell_1\to X$ satisfying $\norm{T}\leq\sup_i\norm{x_i}$. Since $Te_i=x_i$, we have $\Theta Te_i=e_i$ for every $i\in\N$. Since the linear span of $(e_i)_{i\in\N}$ is dense in $\ell_1$, it follows that
\begin{equation*}
    \Theta T=I_{\ell_1}.
\end{equation*}
In particular, $T$ is injective. Set $P=T\Theta$. Then
\begin{equation*}
    P^2=T(\Theta T)\Theta=T\Theta=P.
\end{equation*}
Moreover, $P(X)\subseteq T(\ell_1)$, while $P(Ta)=Ta$ for every $a\in\ell_1$, and hence $P(X)=T(\ell_1)$. Thus $P$ is a bounded projection onto the separable infinite-dimensional subspace $T(\ell_1)$. Consequently, $P\in\Ssep(X)=\Jop(X)\subseteq\K(X)$, which is impossible because a compact projection has finite-dimensional range.
\end{proof}

\begin{remark}
The two quantitative assumptions in \Cref{prop:matrix-obstruction} have distinct roles. The first-row estimate turns the scalar coordinates determined by the matrix units into an $\ell_1$-valued operator, while the bounded first column produces a bounded sequence which is mapped to the unit vector basis. Their combination forces a noncompact operator with separable range.
\end{remark}

\bigskip
\section{The one-ideal algebra}\label{sec:one-ideal-algebra}

We now construct the algebra used in \Cref{thm:main}. Let $D$ be a unital $C^*$-algebra, and put
\begin{equation*}
    E_D=\ell_1(\N,D),
\end{equation*}
with norm
\begin{equation*}
    \norm{x}=\sum_{n \in \N}\norm{x_n}_D
    \qquad
    \bigl(x=(x_n)_{n \in \N} \in E_D\bigr).
\end{equation*}
For $d \in D$ and $i,j \in \N$, define $\theta_{ij}(d) \in \B(E_D)$ by
\begin{equation*}
    \bigl(\theta_{ij}(d)x\bigr)_k=
    \begin{cases}
        dx_j, & k=i,\\
        0, & k\ne i.
    \end{cases}
\end{equation*}
For every $d \in D$ and $i,j \in \N$,
\begin{equation}
\label{eq:theta-norm}
    \norm{\theta_{ij}(d)}=\norm{d}_D.
\end{equation}
Indeed, the inequality $\norm{\theta_{ij}(d)}\leq\norm{d}_D$ follows directly from the $\ell_1$ norm. For the reverse inequality, apply $\theta_{ij}(d)$ to the vector whose $j$th coordinate is $1_D$ and whose remaining coordinates are zero.

The multiplication is given by
\begin{equation}
\label{eq:theta-relations}
    \theta_{ij}(d)\theta_{kl}(e)=\delta_{jk}\theta_{il}(de)
    \qquad(d,e \in D;\ i,j,k,l \in \N).
\end{equation}
Consequently, the linear span of these operators is a subalgebra of $\B(E_D)$. Let
\begin{equation*}
    J_D=\overline{\spanop}\{\theta_{ij}(d):i,j \in \N,\ d \in D\}\subseteq\B(E_D),
\end{equation*}
and define
\begin{equation*}
    A_D=\C I_{E_D}+J_D\subseteq\B(E_D).
\end{equation*}
Thus $J_D$ is the norm closure of the finite-support $D$-valued matrices acting on $E_D$. It is a closed subalgebra of $\B(E_D)$, while $A_D$ is closed because it is the sum of $J_D$ and the finite-dimensional space $\C I_{E_D}$. Moreover, $J_D$ is a closed two-sided ideal of $A_D$. We write
\begin{equation*}
    \varepsilon_{ij}=\theta_{ij}(1_D)
    \qquad(i,j \in \N)
\end{equation*}
for the canonical scalar matrix units.
In particular, $A_D$ is a unital Banach algebra.

The next proposition gives the two structural properties on which the proof rests: $J_D$ is simple and nonunital, and the first diagonal corner of $A_D$ is $D$.

\begin{proposition}[Ideal structure]\label{prop:ideal-structure}
Suppose that $D$ is topologically simple. Then $J_D$ is nonunital and topologically simple. Consequently, $A_D$ has exactly three closed two-sided ideals:
\begin{equation*}
    \{0\}\subsetneq J_D\subsetneq A_D.
\end{equation*}
Moreover,
\begin{equation*}
    \varepsilon_{11}A_D\varepsilon_{11}\cong D
\end{equation*}
as Banach algebras.
\end{proposition}

\begin{proof}
We first prove that $J_D$ is topologically simple. Let $I$ be a non-zero closed two-sided ideal of $J_D$, and choose $0\ne a \in I$. For $n \in \N$, put
\begin{equation*}
    P_n=\sum_{r=1}^{n}\varepsilon_{rr}.
\end{equation*}
The sequence $(P_n)_{n \in \N}$ consists of contractions. If $f$ is a finite-support $D$-valued matrix, then $P_nfP_n=f$ for all sufficiently large $n$. Since
\begin{equation*}
    \norm{P_n(a-f)P_n}\leq\norm{a-f},
\end{equation*}
approximating $a$ by such matrices gives
\begin{equation*}
    P_naP_n\longrightarrow a.
\end{equation*}
Since $a\ne 0$, there are $n,r,s \in \N$, with $r,s\leq n$, such that
\begin{equation*}
    \varepsilon_{rr}a\varepsilon_{ss}\ne 0.
\end{equation*}

We claim that
\begin{equation}
\label{eq:rectangular-corner}
    \varepsilon_{rr}J_D\varepsilon_{ss}=\theta_{rs}(D).
\end{equation}
On finite-support matrices, multiplication by $\varepsilon_{rr}$ and $\varepsilon_{ss}$ selects the $(r,s)$ entry. Hence the left-hand side of \eqref{eq:rectangular-corner} is contained in the closure of $\theta_{rs}(D)$. By \eqref{eq:theta-norm}, the map $d\mapsto\theta_{rs}(d)$ is isometric, so its range is closed. The reverse inclusion is immediate. Thus
\begin{equation*}
    \varepsilon_{rr}a\varepsilon_{ss}=\theta_{rs}(d)
\end{equation*}
for some non-zero $d \in D$.

For every $i,j \in \N$ and every $b,c \in D$, \eqref{eq:theta-relations} gives
\begin{equation*}
    \theta_{ir}(b)\theta_{rs}(d)\theta_{sj}(c)=\theta_{ij}(bdc) \in I.
\end{equation*}
Since $D$ is topologically simple, the closed two-sided ideal generated by $d$ is all of $D$, and hence
\begin{equation*}
    \overline{\spanop}(DdD)=D.
\end{equation*}
The ideal $I$ is closed, and $d\mapsto\theta_{ij}(d)$ is continuous. It follows that
\begin{equation*}
    \theta_{ij}(D)\subseteq I
    \qquad(i,j \in \N).
\end{equation*}
Therefore $I=J_D$, and $J_D$ is topologically simple.

We next show that $J_D$ has no identity. Suppose that $u \in J_D$ were a two-sided identity. Choose a finite-support $D$-valued matrix $f$ such that
\begin{equation*}
    \norm{u-f}<\frac12.
\end{equation*}
Choose $j \in \N$ outside the finite set of column indices occurring in $f$. Then
\begin{equation*}
    f\varepsilon_{jj}=0,
    \qquad
    u\varepsilon_{jj}=\varepsilon_{jj}.
\end{equation*}
Since $\norm{\varepsilon_{jj}}=1$, we obtain
\begin{equation*}
    1=\norm{(u-f)\varepsilon_{jj}}\leq\norm{u-f}<\frac12,
\end{equation*}
a contradiction.

Since $J_D$ is nonunital, $I_{E_D}\notin J_D$, and
\begin{equation*}
    A_D=\C I_{E_D}\oplus J_D
\end{equation*}
is an algebraic and topological direct sum. The map
\begin{equation*}
    \lambda I_{E_D}+j\longmapsto\lambda
\end{equation*}
therefore induces an isomorphism
\begin{equation*}
    A_D/J_D\cong\C.
\end{equation*}

Let $L$ be a non-zero closed two-sided ideal of $A_D$. Suppose first that $L\cap J_D\ne 0$. The simplicity of $J_D$ gives $J_D\subseteq L$. The image of $L$ in $A_D/J_D\cong\C$ is either $0$ or $\C$. In the first case $L=J_D$. In the second case, every element of $A_D$ differs from an element of $L$ by an element of $J_D\subseteq L$, so $L=A_D$.

It remains to exclude $L\cap J_D=0$. Since $L\ne 0$, the ideal $L$ then contains an element
\begin{equation*}
    I_{E_D}+j
\end{equation*}
for some $j \in J_D$, after multiplication by a non-zero scalar. For every $b \in J_D$, both
\begin{equation*}
    (I_{E_D}+j)b
    \qquad\text{and}\qquad
    b(I_{E_D}+j)
\end{equation*}
belong to $L\cap J_D$, and hence vanish. Thus $-j$ is a two-sided identity for $J_D$, a contradiction. The only closed two-sided ideals of $A_D$ are therefore $0$, $J_D$, and $A_D$.

Finally,
\begin{equation*}
    \varepsilon_{11}A_D\varepsilon_{11}=\C\varepsilon_{11}+\varepsilon_{11}J_D\varepsilon_{11} =\varepsilon_{11}J_D\varepsilon_{11} =\theta_{11}(D),
\end{equation*}
because $\varepsilon_{11}=\theta_{11}(1_D)$ already belongs to $\varepsilon_{11}J_D\varepsilon_{11}$. By \eqref{eq:theta-norm}, the map
\begin{equation*}
    \eta\colon D\longrightarrow\varepsilon_{11}A_D\varepsilon_{11},
    \qquad
    \eta(d)=\theta_{11}(d),
\end{equation*}
is an isometric algebra isomorphism.
\end{proof}

We shall also need estimates for the canonical matrix units. They follow directly from the $\ell_1$ norm on $E_D$.

\begin{lemma}[The controlled matrix units in $A_D$]\label{lem:AD-matrix-units}
The family $(\varepsilon_{ij})_{i,j \in \N}\subseteq A_D$ is a system of matrix units, $\varepsilon_{11}\ne 0$, and
\begin{equation*}
    \left\|\sum_{j \in F}a_j\varepsilon_{1j}\right\|
    \leq\sup_{j \in F}|a_j|,
    \qquad
    \sup_{i \in \N}\norm{\varepsilon_{i1}}\leq1
\end{equation*}
for every finite non-empty $F\subseteq \N$ and every scalar family $(a_j)_{j \in F}$.
\end{lemma}

\begin{proof}
The matrix-unit relations follow from \eqref{eq:theta-relations}, and $\varepsilon_{11}\ne0$ by \eqref{eq:theta-norm}. Fix a finite non-empty set $F\subseteq\N$ and a scalar family $(a_j)_{j\in F}$. Let $x=(x_n)_{n\in\N}\in E_D$. The vector
\begin{equation*}
    \left(\sum_{j\in F}a_j\varepsilon_{1j}\right)x
\end{equation*}
has first coordinate $\sum_{j\in F}a_jx_j$ and all other coordinates equal to zero. Hence
\begin{equation*}
    \left\|\left(\sum_{j\in F}a_j\varepsilon_{1j}\right)x\right\|
    =\left\|\sum_{j\in F}a_jx_j\right\|_D \leq\sum_{j\in F}|a_j|\norm{x_j}_D \leq\sup_{j\in F}|a_j|\sum_{j\in F}\norm{x_j}_D\leq\sup_{j\in F}|a_j|\norm{x}.
\end{equation*}
Taking the supremum over all $x\in E_D$ with $\norm{x}\leq1$ gives
\begin{equation*}
    \left\|\sum_{j\in F}a_j\varepsilon_{1j}\right\|\leq\sup_{j\in F}|a_j|.
\end{equation*}
Finally, fix $i\in\N$. The operator $\varepsilon_{i1}$ moves the first coordinate to the $i$th coordinate and annihilates all the other coordinates. Thus
\begin{equation*}
    \norm{\varepsilon_{i1}x}=\norm{x_1}_D\leq\norm{x}
    \qquad(x\in E_D),
\end{equation*}
so $\norm{\varepsilon_{i1}}\leq1$ and hence $\sup_{i\in\N}\norm{\varepsilon_{i1}}\leq1$.
\end{proof}

We can now prove the main theorem.

\begin{proof}[Proof of \Cref{thm:main}]
Let $D$ be the algebra from \Cref{thm:HK}, and put
\begin{equation*}
    A=A_D.
\end{equation*}
By \Cref{rem:Cstar-simple} and \Cref{prop:ideal-structure}, the algebra $A$ has exactly one non-zero proper closed two-sided ideal, namely $J_D$, and the corner
\begin{equation*}
    pAp,
    \qquad
    p=\varepsilon_{11},
\end{equation*}
is isomorphic to $D$.

The map
\begin{equation*}
    A\longrightarrow pAp,
    \qquad
    a\longmapsto pap,
\end{equation*}
is a continuous surjection. Hence
\begin{equation*}
    \mathfrak c=\dens(D)=\dens(pAp)\leq\dens(A).
\end{equation*}
Conversely, choose a dense subset $M\subseteq D$ with $|M|=\mathfrak c$, and put
\begin{equation*}
    \F_0=\mathbb Q+i\mathbb Q.
\end{equation*}
The set of finite $\F_0$-linear combinations of
\begin{equation*}
    I_{E_D}
    \quad\text{and}\quad
    \theta_{ij}(d)
    \qquad(i,j \in \N,\ d \in M)
\end{equation*}
is dense in $A$ and has cardinality at most $\mathfrak c$. Thus
\begin{equation*}
    \dens(A)\leq\mathfrak c,
\end{equation*}
and consequently $\dens(A)=\mathfrak c$.

Fix either choice $\Jop=\A$ or $\Jop=\K$. Suppose, towards a contradiction, that there is a Banach-algebra isomorphism
\begin{equation*}
    \Phi\colon A\longrightarrow\Q(X)
\end{equation*}
for some Banach space $X$. If $X$ is finite-dimensional, then
\begin{equation*}
    \A(X)=\K(X)=\B(X)
\end{equation*}
and $\Q(X)=0$, which is impossible because $A\ne0$. We may therefore assume that $X$ is infinite-dimensional.

First suppose that $X$ is separable. The element $\Phi(p)$ is an idempotent in $\Q(X)$. By \Cref{lem:calkin-corners}, there is a separable complemented subspace $Y\subseteq X$ such that
\begin{equation*}
    D \cong pAp \cong \Phi(p)\Q(X)\Phi(p) \cong\Q(Y).
\end{equation*}
This contradicts \Cref{thm:HK}.

Now suppose that $X$ is nonseparable, and put
\begin{equation*}
    \mathcal R_X=\Ssep(X)/\Jop(X)\subseteq\Q(X).
\end{equation*}
By \Cref{lem:separable-range-ideal}, $\mathcal R_X$ is a proper closed two-sided ideal of $\Q(X)$. Its inverse image under $\Phi$ is therefore a proper closed two-sided ideal of $A$, and so
\begin{equation*}
    \Phi^{-1}(\mathcal R_X)=0
    \qquad\text{or}\qquad
    \Phi^{-1}(\mathcal R_X)=J_D.
\end{equation*}

Suppose first that
\begin{equation*}
    \Phi^{-1}(\mathcal R_X)=J_D.
\end{equation*}
Since $p\in J_D$, we have $\Phi(p)\in\mathcal R_X$. By \Cref{lem:separable-range-corner}, there is a separable complemented subspace $Y\subseteq X$ such that
\begin{equation*}
    D \cong pAp \cong \Phi(p)\Q(X)\Phi(p) \cong\Q(Y).
\end{equation*}
This again contradicts \Cref{thm:HK}.

We are left with
\begin{equation*}
    \Phi^{-1}(\mathcal R_X)=0.
\end{equation*}
Since $\Phi$ is onto, it follows that $\mathcal R_X=0$, and hence
\begin{equation*}
    \Ssep(X)=\Jop(X).
\end{equation*}
By \Cref{lem:AD-matrix-units}, the family $(\Phi(\varepsilon_{ij}))_{i,j \in \N}$ satisfies the matrix-unit relations and has non-zero $(1,1)$ entry. Moreover, for every finite non-empty $F\subseteq \N$ and every scalar family $(a_j)_{j \in F}$,
\begin{equation*}
    \left\|\sum_{j \in F}a_j\Phi(\varepsilon_{1j})\right\|
    \leq\norm{\Phi}\left\|\sum_{j \in F}a_j\varepsilon_{1j}\right\|\leq\norm{\Phi}\sup_{j \in F}|a_j|,
\end{equation*}
and
\begin{equation*}
    \sup_{i \in \N}\norm{\Phi(\varepsilon_{i1})}\leq\norm{\Phi}.
\end{equation*}
Thus the hypotheses of \Cref{prop:matrix-obstruction} hold, a contradiction.

Both possibilities lead to a contradiction. Therefore $A$ is not isomorphic to $\Q(X)$ for any Banach space $X$. Since the choice of $\Jop$ was arbitrary and the algebra $A=A_D$ is independent of that choice, this same $A$ is isomorphic to neither $\B(X)/\A(X)$ nor $\B(X)/\K(X)$ for any Banach space $X$.
\end{proof}

\bigskip

\noindent\textbf{Acknowledgements.} This paper forms part of the first-named author's PhD research at Lancaster University, conducted under the supervision of Professor N. J. Laustsen. The first-named author gratefully acknowledges the support of the EPSRC (grant number EP/W524438/1). We are thankful to Professor Laustsen for his insightful comments and valuable suggestions regarding some points raised in this manuscript. \medskip

\noindent\textbf{AI usage statement.} The idea for the present note arose from previous work of the authors \cite{AcuavivaAcuaviva2026Wild}, and we refer to that paper for a detailed account of AI use. The present paper combines the obstruction found there with an algebra of Horv\'ath and Kania. The authors take full responsibility for the mathematical contents and correctness of the paper. \medskip

For the purpose of open access, the first-named author has applied a Creative Commons Attribution (CC BY) licence to any Author Accepted Manuscript version arising.

\end{document}